\documentclass{amsart2020}

\newtheorem{theorem}{Theorem}
\newtheorem{lemma}[theorem]{Lemma}
\newtheorem{proposition}[theorem]{Proposition}
\newtheorem{corollary}[theorem]{Corollary}

\theoremstyle{definition}

\usepackage{amscd,amssymb}

\begin{document}

\title[Idempotents in matrix rings]
{Idempotents of $2\times 2$ matrix rings\\
over rings of formal power series}

\author[Vesselin Drensky]
{Vesselin Drensky}

\address{Institute of Mathematics and Informatics,
Bulgarian Academy of Sciences,
1113 Sofia, Bulgaria}
\email{drensky@math.bas.bg}

\thanks{Partially supported by
Grant KP-06 N 32/1 of 07.12.2019 ``Groups and Rings -- Theory and Applications'' of the Bulgarian National Science Fund.}

\subjclass[2020]{16U40, 16S50, 15A15, 15B33, 15A54, 13B25, 13F25}
\keywords{idempotents, matrices over rings of power series, Cayley-Hamilton theorem}

\maketitle

\begin{abstract}
Let $A_1,\ldots,A_s$ be unitary commutative rings which do not have non-trivial idempotents
and let $A=A_1\oplus\cdots\oplus A_s$ be their direct sum.
We describe all idempotents in the $2\times 2$ matrix ring $M_2(A[[X]])$
over the ring $A[[X]]$ of formal power series with coefficients in $A$ and in arbitrary set of variables $X$.
We apply this result to the matrix ring $M_2({\mathbb Z}_n[[X]])$ over the ring ${\mathbb Z}_n[[X]]$
for an arbitrary positive integer $n$ greater than 1.
Our proof is elementary and uses only the Cayley-Hamilton theorem (for $2\times 2$ matrices only) and,
in the special case $A={\mathbb Z}_n$, the Chinese reminder theorem and the Euler-Fermat theorem.
\end{abstract}

\section*{Introduction}

In this paper we consider unitary associative rings only.
An element $a$ in a ring $A$ is called an idempotent if $a^2=a$.
The idempotents were introduced in Ring Theory by Benjamin Peirce in his seminal book \cite{P}
(published lithographically in 1870 in a small number of copies for distribution among his friends
and posthumously in its journal form by his son Charles Sanders Peirce in 1881).
There Peirce established the decomposition of rings with idempotents
known nowadays as the Peirce decomposition, see \cite[Proposition 41]{P}.
Already 150 years the study of idempotents is among the main topics in Ring Theory and its applications.
For example the search in the database of Mathematical Reviews
gives more than 2600 publications with the word ``idempotents'' in the title.
For different aspects of the study of idempotents in the classical spirit see for example \'Anh, Birkenmeier and van Wyk \cite{ABvW2}
and the references there.

It is well known that the idempotents of the $d\times d$ matrix algebra $M_d(F)$ over a field $F$
coincide with the diagonalizable matrices with eigenvalues equal to 0 and 1.
In 1946 Foster \cite{F} described the commutative rings $A$ with the property that the idempotents in $M_d(A)$ are diagonalizable for all $d$.
In 1966 Steger \cite{S} showed that important classes of rings have this property.
Among them are polynomial rings in one variable over a principal ideal ring (also with zero divisors)
and polynomial rings in two variables over a $\pi$-regular ring with finitely many idempotents.
(The ring $A$ is $\pi$-regular if for any $a\in A$ there exists an $n$ such that $a^n\in a^nAa^n$.)
The results of Foster and Steger were generalized also for matrices over noncommutative rings, see Song and Guo \cite{SG}.
In \cite{G-TKLN} G\'omez-Torrecillas, Kutas, Lobillo and Navarro presented an algorithm for of computing a primitive idempotent
of a central simple algebra over the field ${\mathbb F}_q(x)$ of rational functions over the  finite field ${\mathbb F}_q$ with applications to coding theory.
In 1967 J.A. Erdos \cite{E} proved that every singular matrix over a field is a product of idempotent matrices.
See also the recent preprint of Nguyen \cite{N} and the references there for further developments in this direction.
We shall mention also the paper
by \'Anh, Birkenmeier and van Wyk \cite{ABvW1} where the authors mimic the behavior of idempotents in matrix rings  in a more general setup
and the following three papers which are related to the present project:
by Birkenmeier, Kim and Park \cite{BKP} and
Kanwar, Leroy and Matczuk \cite{KLM} for relations between the idempotents of $A$, $A[X]$ and $A[[X]]$ where $X$ is a finite set of variables
and by Isham and Monroe \cite{IM} for the properties of the idempotents in ${\mathbb Z}_n$.

It is a natural problem to describe explicitly the idempotents of the matrix rings $M_d(A)$ and $M_d(A[x])$, respectively,
over a commutative ring $A$ and over the polynomial ring $A[x]$ if the idempotents of $A$ are known.
Most of the explicit results in this direction are for the $2\times 2$ matrix ring $M_2({\mathbb Z}_n[x])$
for positive integers $n$ with a small number of prime factors.
For $p$ prime Kanwar, Khatkar and Sharma \cite{KKS} described the idempotents
of $M_2({\mathbb Z}_p[x])$, $M_2({\mathbb Z}_{2p}[x])$ ($p$ odd) and $M_2({\mathbb Z}_{3p}[x])$ ($p>3$).
Balmaceda and Datu \cite{BD} found the idempotents in $M_2({\mathbb Z}_{pq}[x])$ and $M_2({\mathbb Z}_{p^2}[x])$, where $p$ and $q$ are any primes.
Finally, Mittal \cite{M} described the idempotents in $M_2({\mathbb Z}_{pqr}[x])$ for three pairwise different primes greater than 3.
Mittal raised also the question to find the idempotents in $M_2({\mathbb Z}_n[x])$ for any square-free positive integer $n$.

The main step in \cite{KKS} is the description of the idempotents in $M_2(A[x])$
where $A$ is a commutative ring without non-trivial (i.e. different from 0 and 1) idempotents.
Since this holds for $A={\mathbb Z}_p$ for $p$ prime, the authors apply the Chinese remainder theorem and the Euler-Fermat theorem to handle the cases
$A={\mathbb Z}_{2p},{\mathbb Z}_{3p}$. A similar approach was applied in \cite{BD} and \cite{M}.

In the present paper we simplify the ideas in \cite{KKS, BD, M}
and give an explicit presentation of the idempotents of $M_2(A[[X]])$
where $A$ is a direct sum of a finite number of commutative rings without non-trivial idempotents
and $A[[X]]$ is the ring of formal power series over an arbitrary (also infinite) set of commuting variables.
As a consequence we describe the idempotents of $M_2({\mathbb Z}_n[X])$ when $n$ is an arbitrary positive integer greater than 1.
Our proofs are very transparent and use well known elementary arguments only.
They are based on the Cayley-Hamilton theorem (for $2\times 2$ matrices only) and,
as in \cite{KKS, BD, M}, on the Chinese remainder theorem and the Euler-Fermat theorem.

\section{The main result}

We assume that $X$ is an arbitrary set of commuting variables and for a commutative ring $A$ we consider the ring $A[[X]]$
of formal power series in $X$.
In order to make the exposition self-contained we include also the proofs of some well known facts on idempotents.

The following lemma is a partial case of \cite[Proposition 2.5]{BKP} and \cite[Lemma 1]{KLM}.

\begin{lemma}\label{idempotents in power series}
Let $A$ be a commutative ring without non-trivial idempotents. Then the ring $A[[X]]$ also has only trivial idempotents.
\end{lemma}

\begin{proof}
Let $a(X)=a_0+a_1+a_2+\cdots\in A[[X]]$ be an idempotent, where $a_i$ is the homogeneous component of degree $i$ of $a(X)$.
Let $a(X)\notin A$ and $a_1=\cdots=a_{k-1}=0$, $a_k\not=0$.
Since $a^2(X)=a(X)$, comparing the homogeneous components of $a(X)$ and $a^2(X)$ we obtain that
$a_0^2=a_0$ in $A$ and $2a_0a_k=a_k$. Since $A$ has trivial idempotents only, we have that either $a_0=1$ or $a_0=0$.
Both cases are impossible: If $a_0=1$, then $2a_k=a_k$ and $a_k=0$; if $a_0=0$, then again $a_k=0$ which is a contradiction.
Therefore the idempotent $a(X)$ belongs to $A$ and hence is trivial.
\end{proof}

The following proposition is a partial case of \cite[Proposition 3.2]{KKS}
but in virtue of Lemma \ref{idempotents in power series} both statements are equivalent.
We give a simpler straightforward proof.

\begin{proposition}\label{matrix idempotents for rings without non-trivial idempotents}
Let $A$ be a commutative ring without non-trivial idempotents. Then all idempotents in $M_2(A)$ are
\begin{equation}\label{form of the idempotents}
a=\left(\begin{matrix}\alpha&\beta\\ \gamma&1-\alpha\\ \end{matrix}\right),I_2=\left(\begin{matrix}1&0\\0&1\\ \end{matrix}\right),
0_2=\left(\begin{matrix}0&0\\0&0\\ \end{matrix}\right),
\end{equation}
where $\alpha,\beta,\gamma\in A$ and $\alpha(1-\alpha)=\beta\gamma$.
\end{proposition}

\begin{proof}
Let \[
a=\left(\begin{matrix} \alpha&\beta\\ \gamma&\delta\end{matrix}\right)\in M_2(A),\quad \alpha,\beta,\gamma,\delta\in A,
\]
be an idempotent. By the Cayley-Hamilton theorem we have
\[
a^2-\text{tr}(a)a+\det(a)I_2=0_2.
\]
Subtracting the equality $a^2-a=0_2$ we obtain that
\[
(\text{tr}(a)-1)a=\det(a)I_2.
\]
The determinant $\det(a)$ of $a$ is an idempotent of $A$ because the equality $a^2=a$ implies
$\det(a)=\det(a^2)=\det^2(a)$. Hence $\det(a)=1$ or $\det(a)=0$.
First, let $\det(a)=1$. Comparing the entries of the matrices in the equality $(\text{tr}(a)-1)a=0_2$ we obtain that
\[
(\text{tr}(a)-1)\alpha=(\text{tr}(a)-1)\delta=1,(\text{tr}(a)-1)\beta=(\text{tr}(a)-1)\gamma=0.
\]
Hence $(\text{tr}(a)-1)$ is invertible in $A$. This implies that $\beta=\gamma=0$ and $\alpha=\delta\not=0$.
Hence $\alpha I_2=a=a^2=\alpha^2I_2$ and $\alpha$ is an idempotent. Therefore $\alpha=1$ and $a=I_2$.
Now, let $\det(a)=0$. Hence $(\text{tr}(a)-1)a=0_2$ and
\[
(\text{tr}(a)-1)\alpha=(\text{tr}(a)-1)\delta,(\alpha+\delta-1)\alpha=(\alpha+\delta-1)\delta=0,
\]
\[
(-(\alpha+\delta-1))^2=(\alpha+\delta-1)\alpha+(\alpha+\delta-1)\delta-(\alpha+\delta-1)=-(\alpha+\delta-1).
\]
We obtain that $-(\alpha+\delta-1)$ is an idempotent and is equal to 1 or 0. The former case implies that $-a=0_2$ and $a=0_2$.
In the latter case $\delta=1-\alpha$ and $a$ has the desired form (\ref{form of the idempotents}). Since $\det(a)=\alpha\delta-\beta\gamma=0$
we obtain the restriction $\alpha(1-\alpha)=\beta\gamma$. A direct verification shows that all matrices of this form are idempotents.
\end{proof}

\begin{corollary}\label{idempotents for direct sums}
Let $A_1,\ldots,A_s$ be commutative rings without non-trivial idempotents and $A=A_1\oplus\cdots\oplus A_s$ be their direct sum.
Then all idempotents $a(X)\in M_2(A[[X]])$ in the $2\times 2$ matrix ring with entries from $A[[X]]$
are obtained by the following procedure.
We split the set of indices $\{1,\ldots,s\}$ in three parts
\[
P=\{p_1,\ldots,p_k\}, Q=\{q_1,\ldots,q_l\},R=\{r_1,\ldots,r_m\}
\]
and present $A$ in the form $A=A_P\oplus A_Q\oplus A_R$, where
\[
A_P=\bigoplus_{p\in P}A_p, A_Q=\bigoplus_{q\in Q}A_q, A_R=\bigoplus_{r\in R}A_r.
\]
We choose power series $\alpha(X),\beta(X),\gamma(X)\in A_P[[X]]$ such that
$\alpha(X)(1-\alpha(X))=\beta(X)\gamma(X)$. Then $a(X)=(a_P(X),I_2,0_2)$, where $I_2\in M_2(A_Q[[X]])$, $0_2\in M_2(A_R[[X]])$ and
\begin{equation}\label{the P-component}
a_P(X)=\left(\begin{matrix}\alpha(X)&\beta(X)\\ \gamma(X)&1-\alpha(X)\\ \end{matrix}\right)\in M_2(A_P[[X]]).
\end{equation}
\end{corollary}

\begin{proof}
Since each $A_i$, $i=1,\ldots,s$, does not have non-trivial idempotents, by Lemma \ref{idempotents in power series}
the same holds for the rings of power series $A_i[[X]]$.
Applying Proposition \ref{matrix idempotents for rings without non-trivial idempotents}
we obtain that the idempotents in $M_2(A_i[[X]])$ are of the form
\begin{equation}\label{the PQR cases}
a_i(X)=\left(\begin{matrix}\alpha_i(X)&\beta_i(X)\\ \gamma_i(X)&1-\alpha_i(X)\\ \end{matrix}\right),
\quad \alpha_i(X)(1-\alpha_i(X))=\beta_i(X)\gamma_i(X),
\end{equation}
where $\alpha_i(X),\beta_i(X),\gamma_i(X)\in A_i[[X]]$,
or $a_i(X)=I_2$, or $a_i(X)=0_2$.
We present the set $\{1,\ldots,s\}$ as a disjoint union of three subsets $P,Q$ and $R$,
where $p\in P$ if $a_p(X)$ is of the form (\ref{the PQR cases}),
$q\in Q$ if $a_q(X)=I_2^{(q)}$ (the identity matrix in $M_2(A_q[[X]])$)
and $r\in R$ if $a_r(t)=0_2^{(r)}$ (the zero matrix in $M_2(A_r[[X]])$).
Let
\[
a(X)=(a_{p_1}(X),\ldots,a_{p_k}(X))\in A_{p_1}[[X]]\oplus\cdots\oplus A_{p_k}[[X]].
\]
Since $A_P[[X]]=A_{p_1}[[X]]\oplus\cdots\oplus A_{p_k}[[X]]$, we obtain that $a(X)$ has the form (\ref{the P-component}) and
\[
\begin{array}{l}
\alpha(X)=(\alpha_{p_1}(X),\ldots,\alpha_{p_k}(X)),\\
\beta(X)=(\beta_{p_1}(X),\ldots,\beta_{p_k}(X)),\\
\gamma(X)=(\gamma_{p_1}(X),\ldots,\gamma_{p_k}(X))\\
\end{array}
\]
satisfy the relation $\alpha(X)(1-\alpha(X))=\beta(X)\gamma(X)$ because the coordinate power series
$\alpha_p(X),\beta_p(X),\gamma_p(X)$ satisfy $\alpha_p(X)(1-\alpha_p(X))=\beta_p(X)\gamma_p(X)$
for all $p=p_1,\ldots,p_k$. Obviously $(I_2^{(q_1)},\ldots,I_2^{(q_l)})$ equals the identity matrix in $M_2(A_Q[[X]])$ and
$(0_2^{(r_1)},\ldots,0_2^{(r_l)})$ is the zero matrix in $M_2(A_R[[X]])$ which completes the proof.
\end{proof}

We shall apply Corollary \ref{idempotents for direct sums} for $A={\mathbb Z}_n$, $n>1$. We need the following well known fact.
We include the proof for self-completeness of the exposition.

\begin{lemma}\label{idempotents modulo power of prime}
Let $p$ be a prime and $d>0$. Then all the idempotents of the ring ${\mathbb Z}_{p^d}$ are trivial.
\end{lemma}

\begin{proof}
Let $\alpha\in{\mathbb Z}$ be such that its image $\overline{\alpha}$ in ${\mathbb Z}_{p^d}$ is an idempotent.
Hence $\alpha^2-\alpha\equiv 0\, (\text{mod }p^d)$
and $p^d$ divides $\alpha(\alpha-1)$. Since $\alpha$ and $\alpha-1$ are coprime,
we have that either $p^d$ divides $\alpha$ and $\overline{\alpha}=0$ in ${\mathbb Z}_{p^d}$,
or  $p^d$ divides $\alpha-1$ and $\overline{\alpha}=1$ in ${\mathbb Z}_{p^d}$.
\end{proof}

The following theorem was the main motivation to start the present project.

\begin{theorem}\label{main theorem}
Let $n>1$ be a positive integer. Then all idempotents $a(X)$ in $M_2({\mathbb Z}_n[[X]])$ are obtained in the following way.
We present $n$ as a product $n=PQR$ of three pairwise coprime positive integers $P,Q,R$. If $P>1$ we choose three power series
$\alpha(X),\beta(X),\gamma(X)\in{\mathbb Z}[[X]]$ such that $\alpha(X)(1-\alpha(X))\equiv \beta(X)\gamma(X)\, (\text{\rm mod }P)$.
Then modulo $n$
\[
a(X)\equiv \left(\begin{matrix}\overline{\alpha}(X)&\overline{\beta}(X)\\
\overline{\gamma}(X)&1-\overline{\alpha}(X)\\ \end{matrix}\right),
\]
where:

{\rm (i)} If $P,Q,R>1$, then
\[
\overline{\alpha}(X)\equiv(\alpha(X)+(1-\alpha(X))P^{\varphi(Q)})(1-(PQ)^{\varphi(R)}),
\]
\[
\overline{\beta}(X)\equiv\beta(X)(1-P^{\varphi(Q)\varphi(R)}),
\overline{\gamma}(X)\equiv\gamma(X)(1-P^{\varphi(Q)\varphi(R)});
\]

{\rm (ii)} If $P,Q>1$, $R=1$, then
\[
\overline{\alpha}(X)\equiv\alpha(X)+(1-\alpha(X))P^{\varphi(Q)},
\overline{\beta}(X)\equiv\beta(X)(1-P^{\varphi(Q)}),
\overline{\gamma}(X)\equiv\gamma(X)(1-P^{\varphi(Q)});
\]

{\rm (iii)} If $P,R>1$, $Q=1$, then
\[
\overline{\alpha}(X)\equiv\alpha(X)(1-P^{\varphi(R)}),\overline{\beta}(X)\equiv\beta(X)(1-P^{\varphi(R)}),
\overline{\gamma}(X)\equiv\gamma(X)(1-P^{\varphi(R)});
\]

{\rm (iv)} If $P=1$, $Q,R>1$, then $\overline{\alpha}(X)\equiv 1-Q^{\varphi(R)}$, $\overline{\beta}(X)\equiv\overline{\gamma}(X)\equiv 0$;

{\rm (v)} If $P>1$, $Q=R=1$, then $\overline{\alpha}(X)\equiv \alpha(X)$, $\overline{\beta}(X)\equiv \beta(X)$,
$\overline{\gamma}(X)\equiv \gamma(X)$;

{\rm (vi)} If $P=R=1$, $Q>1$, then $a(X)\equiv I_2$;

{\rm (vii)} If $P=Q=1$, $R>1$, then $a(X)\equiv 0_2$,

\noindent and $\varphi$ is the Euler totient function.
\end{theorem}

\begin{proof}
As in \cite[page 151]{KKS}, if $n=\prod  p^d$, where $p$ are the prime divisors of $n$,
we present the ring ${\mathbb Z}_n$ as the direct sum of the rings ${\mathbb Z}_{p^d}$.
If $a(X)\in M_2({\mathbb Z}_n[[X]])$ is an idempotent, by Lemma \ref{idempotents modulo power of prime} we can apply Corollary \ref{idempotents for direct sums}.
We divide the prime divisors of $n$ in three groups
$\{p_1,\ldots,p_k\}$, $\{q_1,\ldots,q_l\}$ and $\{r_1,\ldots,r_m\}$ depending on the form of the projection of $a(X)$ in $M_2({\mathbb Z}_{p^d}[[X]])$:
$p\in\{p_1,\ldots,p_k\}$ if $a(X)$ is of the form (\ref{the PQR cases}) and $p\in \{q_1,\ldots,q_l\}$ or $p\in\{r_1,\ldots,r_m\}$ if the projection
is, respectively, the identity matrix and the zero matrix. Let $P=p_1^{d_1}\cdots p_k^{d_k}$, $Q=q_1^{e_1}\cdots q_l^{e_l}$, $R=r_1^{f_1}\cdots r_m^{f_m}$.
The image of $a(X)$ in $M_2({\mathbb Z}_P[[X]])\cong M_2({\mathbb Z}_{p_1^{d_1}}[[X]])\oplus\cdots\oplus M_2({\mathbb Z}_{p_k^{d_k}}[[X]])$ is of the form (\ref{the PQR cases}).
If we choose the images of $\alpha(X),\beta(X),\gamma(X)$ modulo $p_i^{d_i}$, we can find their images modulo $P$ using the Chinese reminder theorem.
The images of $a(X)$ in $M_2({\mathbb Z}_Q[[X]])\cong M_2({\mathbb Z}_{q_1}[[X]])\oplus\cdots\oplus M_2({\mathbb Z}_{q_l}[[X]])$ and
$M_2({\mathbb Z}_R[[X]])\cong M_2({\mathbb Z}_{r_1}[[X]])\oplus\cdots\oplus M_2({\mathbb Z}_{r_m}[[X]])$ are, respectively, equal to the identity matrix and the zero matrix.
Since $P,Q$ and $R$ are pairwise coprime, it is sufficient to check whether the form of $a(X)$ given in the cases (i) -- (vii) in the theorem
satisfy the required conditions modulo $P, Q$ and $R$. We shall check this for $\overline{\alpha}(X)$ in the case (i) only.
The other cases are handled similarly. Since $\varphi(Q),\varphi(R)\geq 1$, obviously
\[
\overline{\alpha}(X)\equiv(\alpha(X)+(1-\alpha(X))P^{\varphi(Q)})(1-(PQ)^{\varphi(R)}\equiv \alpha(X)\,(\text{mod }P).
\]
By the Euler-Fermat theorem $P^{\varphi(Q)}\equiv 1\,(\text{mod }Q)$. Hence
\[
\overline{\alpha}(X)\equiv(\alpha(X)+(1-\alpha(X))\equiv 1\,(\text{mod }Q),
\]
and in a similar way we establish that $\overline{\alpha}(X)\equiv0\,(\text{mod }R)$.
\end{proof}

\end{document}